\documentclass[10pt]{article}
\usepackage{geometry}                
\usepackage{graphicx}
\usepackage{amssymb, amstext, amsthm, amsfonts, amsmath}
\usepackage{mathrsfs}
\usepackage{enumerate}
\usepackage[noadjust]{cite}
\usepackage{color}
\usepackage{hyperref}

\newcommand{\Q}{\mathbb{Q}} 
\newcommand{\Z}{\mathbb{Z}} 

\newcommand{\qhat}{\widehat{q}}
\newcommand{\qinv}{q^{-1}}

\newcommand{\twopartdef}[4]
{
	\left\{
		\begin{array}{ll}
			#1 & \mbox{if } #2 \\
			#3 & \mbox{if } #4 
		\end{array}
	\right.
}

\newcommand{\threepartdef}[6]
{
	\left\{
		\begin{array}{ll}
			#1 & \mbox{if } #2 \\
			#3 & \mbox{if } #4 \\
			#5 & \mbox{if } #6
		\end{array}
	\right.
}
\newcommand{\co}{\mathrm{co\;}}
\newcommand{\im}{\mathrm{im\;}}

\newcommand{\SL}{\mathrm{SL}}
\newcommand{\frsl}{\mathfrak{sl}}
\newcommand{\Sym}{\mathrm{Sym}}

\newtheorem{thm}{Theorem}
\newtheorem{lem}[thm]{Lemma}
\newtheorem{cor} [thm]{Corollary}

\title{Quantized Coordinate Rings of the  Unipotent Radicals of the Standard Borel Subgroups in $\SL_{n+1}$\footnote{The results of this paper will appear in the author's Ph.D. thesis at the University of California, Santa Barbara}}
\author{Andrew Jaramillo}
\date{\today}

\begin{document}
\maketitle
\section{Introduction}
 Since their discovery in the 1980's by the Leningrad school, there has been much work done on ``quantum groups". Still though there is not yet a widely accepted axiomatic definition for them. (\cite{BroGoo2002} Problem II.10.1) Nevertheless, there are many objects which are referred to as quantum groups  or quantized rings of classical objects.  Many of them share the property that they are noncommutative deformations of classical $k$-algebras (in some sense) with a parameter $q$, that when $q$ takes on such special values (e.g. $q=1$) the classical $k$-algebra structure is recovered.   Of course this is not a rigorous definition, but still is indicative of the what we wish to capture when quantizing  some object.

\subsection{The Quantized Coordinate Ring for $\SL_{n+1}$}
With this description in mind we begin by giving a well-known presentation for the quantized coordinate ring for $\SL_{n+1}$.

We adopt the following conventions: $k$ is a field  and $q \in k^\times$ with $q$ not  a root of unity unless  noted otherwise. Define $\widehat{q}=q-q^{-1}$.
Following   I.2.2 in \cite{BroGoo2002} we define $O_q(\SL_{n+1})$ as the $k$-algebra  generated by  $\{X_{ij}\;|\; 1\leq i,j\leq n+1\}$ presented with the following relations:

\begin{align}
X_{ij}X_{im}&=qX_{im}X_{ij} \;\mathrm{ for }\; j<m\label{OqM1}\\
X_{ij}X_{lj}&=qX_{lj}X_{ij}\;\mathrm{ for }\; i<l \label{OqM2}\\
X_{ij}X_{lm}&=X_{lm}X_{ij}\;\mathrm{ for}\; i<l \;\mathrm{and}\; j>m\label{OqM3}\\
X_{ij}X_{lm}-X_{lm}X_{ij}&=\qhat X_{im}X_{lj}\; \mathrm{ for}\; i<l\;\mathrm{ and}\; j<m\label{OqM4}\\
\sum_{\sigma \in \Sym_{n+1}}(-q)^{\ell(\sigma)}&X_{1\sigma(1)}\cdots X_{n+1\sigma(n+1)}=1\label{OqM5}
\end{align}
We note that relations \eqref{OqM1}--\eqref{OqM4} are the defining relations for $O_q(M_{n+1})$ (the presentation is from \cite{Par1991} section 3.5 but replacing $q^{-1}$ there with $q$ in our relations here).
The left hand side of the fifth relation is often referred to as the {\em quantum determinant} for $O_q(\mathrm{M}_{n+1})$. Thus  we say we are ``setting the quantum determinant equal to 1" in $O_q(\SL_{n+1})$.  Note that when $q=1$ we recover exactly the usual presentation for $O(\SL_{n+1}).$ 

\subsection{Quantized Unipotent Radicals of Standard Borel Subgroups in  $\SL_{n+1}$ } 
When studying subgroups of $\SL_{n+1}$ one often begins with the  Borel subgroups as well as the unipotent radicals of such subgroups.  For instance in $\SL_{n+1}$ the {\em positive (resp. negative) standard Borel subgroup} $B^+$ (resp. $B^-$) is the subgroup consisting of all  upper (resp. lower) triangular matrices in $\SL_{n+1}$.  The  {\em positive (resp. negative) unipotent radical $N^+$}  (resp. $N^-$) of $B^+$ (resp $B^-$) is the subgroup of  upper (resp. lower) triangular unipotent matrices in $B^+$ (resp. $B^-$). 

Both $B^\pm$ are  closed subvarieties of  $\SL_{n+1}$.  It follows then that for the coordinate rings  $O(B^\pm)$ we have
\begin{align*}
O(B^+)\cong O(\SL_{n+1})/\langle X_{ij} \mid i>j\rangle& &
O(B^-)\cong O(\SL_{n+1})/\langle X_{ij}\mid i<j\rangle
\end{align*}
Similarly, since $N^\pm$ are  closed subvarieties of $\SL_{n+1}$, for  the coordinate rings  $O(N^\pm)$ we have
$$O(N^\pm)\cong O(B^\pm)/\langle X_{ii}-1\;|\; i=1,\ldots, n \rangle$$

We now wish to try and ``quantize" these coordinate rings. By our  discussion, the quantized coordinate rings,  $O_q(B^\pm)$, are to be  noncommutative deformations  of the coordinate rings of the $B^\pm$.  
 Since $B^\pm$ are Poisson-algebraic subgroups of $\SL_{n+1}$ we also require the semiclassical limits, as in \cite{Goo2008}, of $O_q(B^\pm)$ to be $O(B^\pm)$ as Poisson-algebras.  Having this property does not leave us much choice in defining $O_q(B^\pm)$. Thus following  \cite{Par1991} section 6.1 we define the {\em quantized coordinate rings of the standard Borel Subgroups} to be
\begin{align*}
O_q(B^+):=O_q(\SL_{n+1})/\langle X_{ij}\mid i>j\rangle& &
O_q(B^-):=O_q(\SL_{n+1})/\langle X_{ij}\mid i<j\rangle
\end{align*}
We will often abuse notation and denote simply by  $X_{ij}$  the coset containing $X_{ij}$.

 Attempting to define the quantized coordinate rings  $O_q(N^\pm)$ as 
 $$O_q(N^\pm)=O_q(B^\pm)/\langle X_{ii}-1\mid 1\leq i\leq n+1 \rangle$$
  would not be helpful since relation (1)   in $O_q(N^\pm)$ would imply that for all $i\neq j$
  $$X_{ij}=X_{ii}X_{ij}=qX_{ij}X_{ii}=qX_{ij}$$
 Since $q$ is not a root of unity this implies $X_{ij}=0$ for all $i\neq j$.  Thus $O_q(N^\pm)\cong k$ \cite{Par1991} remark 6.3.
 Though this may be  a nice algebra to study it is not a particularly useful analogue to the classical setting.  Therefore we must try and define $O_q(N^\pm)$ in another way. 
 
 There are (surprisingly) few definitions found in the literature although some authors (\cite{Gei2011}) have defined $O_q(N^\pm)$ to be $U_q^\pm(\frsl_{n+1})$ (see below) since when $q=1$ we recover $O(N^\pm)$.  
This is definition ``quantizes" $O(N^\pm)$  however,
 we  would like to find an algebra more directly related to $O_q(B^\pm)$. Since $N^\pm$ are {\bf not} Poisson-algebraic subgroups of $B^\pm$ there are fewer requirements when defining quantized coordinate rings for them.  
In contrast to  the classical case however, there are no ``natural" quotient algebras of $O_q(B^\pm)$ that reduce to $O(N^\pm)$ when $q=1$. (\cite{Par1991} remark 6.3.) Nevertheless, there exist candidate  subalgebras of $O_q(B^\pm)$ that do have this property, some of which were defined in \cite{Goo2001}, and thus are good candidates for the definition of $O_q(N^\pm)$.  

Another possible way to define $O_q(N^\pm)$ is to first note  that  $B^\pm$ is the semidirect product $T \ltimes N^\pm$ where $T$ is the diagonal subgroup (standard maximal torus) of $\SL_{n+1}$.  Thus $O(N^\pm)$ is the algebra of coinvariants using the corresponding coaction of $O(T)$ on $O(B^\pm)$.  Therefore another possible definition for $O_q(N^\pm)$ is as the subalgebra of coinvariants for a suitable coaction of $O_q(T)$ on $O_q(B^\pm)$ using the Hopf algebra structure of $O_q(\SL_{n+1})$.

We shall prove that all the the above approaches for defining $O_q(N^\pm)$ are isomorphic.  We begin with those subalgebras defined in  \cite{Goo2001} since these are more easily described and related. Finally, once we have established the isomorphisms   we apply the isomorphisms $U_q^{\geq 0} \cong O_q(B^+)$ and $U_q^{\leq 0} \cong O_q(B^-)$ established by Joseph in \cite{Jos1995} and show the above definitions for $O_q(N^\pm)$ are  isomorphic to $U_q^\pm(\frsl_{n+1})$.

\section{Definition and Structure of $O_q(N^\pm)$ and $O_q(N^\pm)'$}
Note that in $O_q(B^+ )$ (resp. $O_q(B^-)$) that  $X_{ij} =0$ for $i>j$ (resp. $i<j$).  Moreover,   since the quantum determinant is $1$,  relation (5)  in $O_q(B^\pm)$ simplifies to
$$X_{11}\cdots X_{n+1,n+1}=1$$ 
 Moreover, by relation (4) in $O_q(B^\pm)$ we have
$X_{ii}X_{jj}=X_{jj}X_{ii}$.
 Taking these two facts together we can conclude that for all $i=1,2,\ldots, n+1$ the elements $X_{ii}$ are in fact, invertible in $O_q(B^\pm)$. Therefore we may define in $O_q(B^+)$ two subalgebras 
$$O_q(N^+):=k\left\langle X_{ii}^{-1}X_{ij} \mid 1\leq i<j\leq n+1\right\rangle$$ 
 $$O_q(N^+)':=k\left\langle X_{ij}X_{jj}^{-1} \mid 1\leq i<j\leq n+1\right \rangle$$
 These are natural choices because they both become isomorphic to $O_q(N^+)$ when $q=1$.  We may analogously define subalgebras in $O_q(B^-)$ by
 $$O_q(N^-):=k[\left\langle X_{jj}^{-1}X_{ji} \mid 1\leq i<j\leq n+1\right\rangle$$ 
 $$O_q(N^-)':=k\left\langle X_{ji}X_{ii}^{-1} \mid 1\leq i<j\leq n+1\right\rangle$$ 
Having defined the algebras $O_q(N^\pm)$ and $O_q(N^\pm)'$, we now analyze their structure.

 \begin{lem}
 \label{StructureOfOqN}
 Define for $1\leq i<j\leq n+1$ the elements $y_{ij}=X_{ii}^{-1}X_{ij}$ and $z_{ij}=X_{ij}X_{jj}^{-1}$ in $O_q(B^+)$. 
 The following are defining relations for $O_q(N^+)$ and $O_q(N^+)'$ respectively:

\begin{align}
y_{ij}y_{im}&=qy_{im}y_{ij}  & &(j<m) \label{Oqn+1}\\
y_{ij}y_{lj}&=qy_{lj}y_{ij} & & (i<l) \label{Oqn+2}\\
y_{ij}y_{lm}&=y_{lm}y_{ij}& &(i<l,\, j>m) \label{Oqn+3}\\
y_{ij}y_{lm}&=\threepartdef{
y_{lm}y_{ij}}{j<l}
{q^{-1}y_{lm}y_{ij}+\qinv \qhat y_{im}}{j=l}
{y_{lm}y_{ij}+\qhat y_{im}y_{lj}}{j>l}& &(i<l,\, j<m) \label{Oqn+4}
\end{align}

 \begin{align}
z_{ij}z_{im}&=qz_{im}z_{ij}  & &(j<m) \label{Oqn'+1}\\
z_{ij}z_{lj}&=qz_{lj}z_{ij} & & (i<l) \label{Oqn'+2}\\
z_{ij}z_{lm}&=z_{lm}z_{ij}& &(i<l,\, j>m) \label{Oqn'+3}\\
z_{ij}z_{lm}&=\threepartdef{
z_{lm}z_{ij}}{j<l}
{q^{-1}z_{lm}z_{ij}+q^{-1}\qhat z_{im}}{j=l}
{z_{lm}z_{ij}+\qhat z_{im}z_{lj}}{j>l}
& &(i<l,\, j<m)  \label{Oqn'+4}
\end{align}
 \end{lem}
 \begin{proof}

 
 First we show that the generators $y_{ij}$ of $O_q(N^+)$ satisfy the relations \eqref{Oqn+1} -- \eqref{Oqn+4} above. Note that from the relations \eqref{OqM3} and \eqref{OqM4} in $O_q(\SL_{n+1})$,  the elements $X_{ii}$ commute with $X_{lm}$  in $O_q(B^+)$ whenever $l,m\neq i$.
 
 For $j<m$ and $i=l$
 \begin{align*}
y_{ij}y_{im}&=X_{ii}^{-1}X_{ij}X_{ii}^{-1}X_{im}=\qinv X_{ii}^{-1}X_{ij}X_{im}X_{ii}^{-1} \\
 &=X_{ii}^{-1}X_{im}X_{ij}X_{ii}^{-1}=qX_{ii}^{-1}X_{im}X_{ii}^{-1}X_{ij}
 =qy_{im}y_{ij}
\end{align*}

For $i<l$ and $j=m$
\begin{align*}
y_{ij}y_{lj}&=X_{ii}^{-1}X_{ij}X_{ll}^{-1}X_{lj}=X_{ii}^{-1}X_{ll}^{-1}X_{ij}X_{lj}=qX_{ii}^{-1}X_{ll}^{-1}X_{lj}X_{ij}\\
&=qX_{ll}^{-1}X_{ii}^{-1}X_{lj}X_{ij}=qX_{ll}^{-1}X_{lj}X_{ii}^{-1}X_{ij}=qy_{lj}y_{ij}
\end{align*}

For $i<l$ and $j>m$
\begin{align*}
y_{ij}y_{lm}&=X_{ii}^{-1}X_{ij}X_{ll}^{-1}X_{lm}=X_{ii}^{-1}X_{ll}^{-1}X_{ij}X_{lm}=X_{ii}^{-1}X_{ll}^{-1}X_{lm}X_{ij}\\
&=X_{ll}^{-1}X_{ii}^{-1}X_{lm}X_{ij}=X_{ll}^{-1}X_{lm}X_{ii}^{-1}X_{ij}=y_{lm}y_{ij}
\end{align*}

For $i<l,\, j<m,$ and $j<l$
\begin{align*}
y_{ij}y_{lm}&=X_{ii}^{-1}X_{ij}X_{ll}^{-1}X_{lm}=X_{ii}^{-1}X_{ll}^{-1}X_{ij}X_{lm}=X_{ii}^{-1}X_{ll}^{-1}X_{lm}X_{ij}\\
&=X_{ll}^{-1}X_{ii}^{-1}X_{lm}X_{ij}=X_{ll}^{-1}X_{lm}X_{ii}^{-1}X_{ij}= y_{lm}y_{ij}
\end{align*}

For $i<l,\, j<m,$ and $j=l$
\begin{align*}
y_{ij}y_{lm}&=X_{ii}^{-1}X_{ij}X_{ll}^{-1}X_{lm}\\
&=\qinv X_{ii}^{-1}X_{ll}^{-1}X_{ij}X_{lm}=\qinv X_{ii}^{-1}X_{ll}^{-1}(X_{lm}X_{ij}+\qhat X_{im}X_{lj})\\
&=\qinv X_{ii}^{-1}X_{ll}^{-1}X_{lm}X_{ij}+\qinv \qhat X_{ii}^{-1}X_{ll}^{-1}X_{im}X_{lj}\\
&=\qinv X_{ll}^{-1}X_{ii}^{-1}X_{lm}X_{ij}+\qinv \qhat X_{ii}^{-1}X_{im}X_{ll}^{-1}X_{lj}\\
&=\qinv X_{ll}^{-1}X_{lm}X_{ii}^{-1}X_{ij}+\qinv \qhat X_{ii}^{-1}X_{im}=q^{-1}y_{lm}y_{ij}+\qinv \qhat y_{im}
\end{align*}

For $i<l,\, j<m,$ and $j>l$
\begin{align*}
y_{ij}y_{lm}&=X_{ii}^{-1}X_{ij}X_{ll}^{-1}X_{lm}= X_{ii}^{-1}X_{ll}^{-1}X_{ij}X_{lm}= X_{ll}^{-1}X_{ii}^{-1}X_{ij}X_{lm}\\
&= X_{ll}^{-1}X_{ii}^{-1}(X_{lm}X_{ij}+\qhat X_{im}X_{lj})= X_{ll}^{-1}X_{ii}^{-1}X_{lm}X_{ij}+\qhat X_{ll}^{-1}X_{ii}^{-1}X_{im}X_{lj}\\
&= X_{ll}^{-1}X_{lm}X_{ii}^{-1}X_{ij}+\qhat X_{ii}^{-1}X_{im}X_{ll}^{-1}X_{lj}=y_{lm}y_{ij}+\qhat y_{im}y_{lj}
\end{align*}

We now show that  the above relations are  a defining set of relations for $O_q(N^+)$.  
Let $B$ be the algebra generated by $\{b_{ij}\mid 1\leq i<j\leq n+1\}$  presented with relations analogous to those in  \eqref{Oqn+1}-\eqref{Oqn+4} above but replacing $y_{ij}$ with $b_{ij}$. Let $\psi$ be the $k$-algebra homomorphism $\psi: B \to O_q(N^+)$ defined by $\psi(b_{ij})=y_{ij}$.

Order the $X_{ij}$ lexicographically in $O_q(B^+)$ omitting $X_{n+1,n+1}$.  As asserted in  the proof of \cite{Goo2001} Lemma 2.8,   the set of monomials ordered in this way is linearly independent.  

This is still true if we allow ordered monomials with negative exponents on $X_{ii}$.  Moreover since the $X_{ii}$ commute up to scalars with every $X_{lm}$ in $O_q(B^+)$ then the set of ordered monomials in the $y_{ij}$ is linearly independent, hence forms a  basis for $O_q(N^+)$.

Now the monomials in the $b_{ij}$ form a spanning set for $B$. 
Hence $\psi$ maps a spanning set of $B$ to a basis of $O_q(N^+)$.  Therefore $\psi$ is an isomorphism.

 It can be similarly verified that the relations \eqref{Oqn'+1}-\eqref{Oqn'+4} give a presentation of  $O_q(N^+)'$.  
 \end{proof}
 
 \begin{thm}
 \label{OqIsomToOq'}
 The algebras $O_q(N^\pm)$ and $O_q(N^\pm)'$ are all isomorphic.
 \end{thm}
 
 \begin{proof}
From Lemma~\ref{StructureOfOqN} it is immediate  that $O_q(N^+)\cong O_q(N^+)'$  since the algebras have the same presentation. 

From \cite{Par1991} Proposition 3.7.1 there exists a transpose homomorphism $\tau:O_q(\SL_{n+1})\to O_q(\SL_{n+1})$ so that $\tau(X_{ij})=X_{ji}$ for all $i,j \in \{1,2,\ldots, n+1\}$ . This is an automorphism of $O_q(\SL_{n+1})$ that maps $\left \langle X_{ij}\mid i>j\right\rangle$ onto $\left\langle X_{ij}\mid i<j\right\rangle$.  Therefore there is an  induced  isomorphism $\overline{\tau}: O_q(B^+)\to O_q(B^-)$. 

Observe that 
\begin{align*}
\overline{\tau}(X_{ii}^{-1}X_{ij})&=X_{ii}^{-1}X_{ji}=q^{-1}X_{ji}X_{ii}^{-1}\\
\overline{\tau}(X_{ij}X_{jj}^{-1})&=X_{ji}X_{jj}^{-1}=q^{-1}X_{jj}^{-1}X_{ji}
\end{align*}
for all $i<j$. 
Hence, $\overline{\tau}$ maps $O_q(N^+)$ onto $O_q(N^-)'$ and so  $O_q(N^+)\cong O_q(N^-)'$.  Similarly $\overline{\tau}$ maps $O_q(N^+)'$ onto $O_q(N^-)$.  Therefore $O_q(N^+)'\cong O_q(N^-)$
 \end{proof}
 
Using this theorem we will refer to $O_q(N^+)$ as the {\em positive quantized unipotent subgroup of $O_q(\SL_{n+1})$}. Similarly, $O_q(N^-)$ is the {\em negative quantized unipotent subgroup.} Notice that if $q=1$ we have a commutative $k$-algebra which matches the usual presentation for $O(N^\pm)$.

\section{Coinvariants of $O_q(B^\pm)$}
A {\em left (resp. right) coinvariant} for a left (resp. right)  comodule algebra $A$ over a Hopf algebra $H$ is any element $x\in A$ so that $\eta(x)=1\otimes x$ (resp. $x\otimes 1$) where $\eta$ is the comodule structure map.  We denote the subalgebra of left (or right) coinvariants in $A$  by $A^{\co \eta}$.

 In \cite{Kli1997} 9.2.3 Proposition 10 it is shown that  $O_q(\SL_{n+1})$ is a Hopf algebra.  Specifically, the comultiplication and the counit are defined by
 \begin{align*}
\Delta(X_{ij}):=\sum_{k=1}^{n+1} X_{ik}\otimes X_{kj}& & \epsilon(X_{ij}):=\delta_{ij}
\end{align*}
 To define the antipode we first recall the notation for a $k \times k$ quantum minor  in $O_q(M_{n+1})$.  That is, for $I \subset \{1,\ldots,n+1\}$ and $J\subset \{1,\ldots,n+1\}$  where  $|I|=|J|=k$  and $I=\{i_1,\ldots,i_k\}$ with $i_1<i_2<\cdots<i_k$ and $J=\{j_1,\ldots,j_k\}$ with $j_1<j_2<\cdots<j_k$ then 
$$[I\mid J]:=\sum_{\sigma \in \Sym_k}(-q)^{\ell(\sigma)}X_{i_1,j_{\sigma(1)}}\cdots X_{i_k,j_{\sigma(k)}}
$$
Denote by $M_{ij}$ the $n \times n$ quantum minor $[1,2,\cdots,\widehat{i},\cdots, n+1\mid 1,2,\cdots, \widehat{j},\cdots, n+1]$ in $O_q(\SL_{n+1})$.   The antipode, $S$,  for $O_q(\SL_{n+1})$ is defined by
\[
S(X_{ij})=q^{i-j}M_{ji}
\]
\subsection{$O_q(B^\pm)$ as  Hopf Algebras}
A {\em Hopf ideal} $I$ in a Hopf algebra $A$ is  an ideal $I\subset \ker \epsilon$ so that   $\Delta (I)\subset A \otimes I +I\otimes A$ and $S(I)\subset I$. If $I$ is a Hopf ideal then $A/I$ is also  a Hopf  algebra induced from the Hopf algebra structure of $A$ using the natural induced comultiplication, counit, and antipode.

Define $I^-:=\langle X_{ij}\mid  i>j\rangle$ and $I^+:=\langle X_{ij}\mid i<j\rangle$ in $O_q(\SL_{n+1})$.   
It follows from Lemma 6.1.1 \cite{Par1991}  that $I^\pm$ are Hopf Ideals
%
%
%
Since $O_q(B^\pm)=O_q(\SL_{n+1})/I^\mp$  then $O_q(B^\pm)$ are Hopf algebras induced from the Hopf algebra structure of $O_q(\SL_{n+1})$.  
We denote the comultiplication, counit, and antipode of the Hopf algebra of $O_q(B^\pm)$ by $\Delta_{B^\pm}$, $\epsilon_{B^\pm}$, and $S_{B^\pm}$ respectively, when emphasis is needed, otherwise we will retain the standard notation $\Delta, \epsilon$, and $S$.
Specifically we note that for $X_{ij} \in O_q(B^+)$ and $X_{rs} \in O_q(B^-)$ we have  
\begin{align*}
\Delta_{B^+}(X_{ij})=\sum_{i\leq k\leq j}X_{ik}\otimes X_{kj}& &\Delta_{B^-}(X_{rs})=\sum_{r\geq k \geq s}X_{rk}\otimes X_{ks}
\end{align*}

\subsection{$O_q(T)$}
In the classical setting the coordinate ring of the standard maximal torus of $O(\SL_{n+1})$ is 
\[
O(T)\cong O(\SL_{n+1})/\langle X_{ij}\;|\; i\neq j\rangle
\]
We may therefore define the {\em quantized coordinate ring of the standard maximal torus for $SL_{n+1}$} by
\[
O_q(T):=O_q(\SL_{n+1})/\langle X_{ij}\:|\; i\neq j\rangle
\]
Denote by $Y_{ii}$ the coset containing $X_{ii}$ in $O_q(T)$.  We first note that $O_q(T)$ is generated by the $Y_{ii}$ where $i=1,2,\ldots,n+1$.
It is straightforward to check that  $Y_{ii}Y_{jj}=Y_{jj}Y_{ii}$ for all $i,j \in \{1,2,\cdots,n+1\}$. Thus $O_q(T)$ is actually a commutative algebra. 
Moreover since the quantum determinant is 1 in $O_q(\SL_{n+1})$ this implies that  
$$Y_{1,1}\cdots Y_{n+1,n+1}=1$$
That is, each of the $Y_{ii}$ are invertible. Therefore we have, in fact, $O_q(T)\cong O(T)$. Moreover, since $O_q(T)=O_q(\SL_{n+1})/(I^++I^-)$ this implies that $O_q(T)$ is also a Hopf algebra induced from $O_q(\SL_{n+1})$. We denote the comultiplication, counit, and antipode by $\Delta_T$, $\epsilon_T$, and $S_T$ respectively, when emphasis is needed. Specifically, we note that for $Y_{ii} \in O_q(T)$ we have
\[
\Delta_T(Y_{ii})=Y_{ii}\otimes Y_{ii}
\]

\subsection{$O_q(T)$-coactions on $O_q(B^\pm)$}
There are  natural projection homomorphisms $p^\pm:O_q(B^\pm) \to O_q(T)$  defined by $p^\pm(X_{ii})=Y_{ii}$ and $p^\pm(X_{ij})=0$ for $i\neq j$.  Therefore using these maps as well as  the comultiplication maps $\Delta$ on $O_q(B^\pm)$ we define the maps $\eta^\pm:O_q(B^\pm)\to O_q(B^\pm)\otimes O_q(T)$ by 
\[
\eta^\pm:=(id \otimes p^\pm) \Delta 
\]
Since $p^\pm$ and $\Delta$ are  $k$-algebra homomorphisms, so are  $\eta^\pm$. Similarly the  maps $\theta^\pm: O_q(B^\pm) \to O_q(T)\otimes O_q(B^\pm)$ defined by   $\theta^\pm:=(p^\pm\otimes id) \Delta$ are $k$-algebra homomorphisms. 
Specifically, since  $p^\pm(X_{ij})=0$ for $i\neq j$ and $p^\pm(X_{ii})=Y_{ii}$ 
 this implies  that for $X_{ij} \in O_q(B^\pm)$ we have
\begin{align*}
 \eta^\pm(X_{ij})=X_{ij}\otimes Y_{jj}& &\theta^\pm(X_{ij})=Y_{ii}\otimes X_{ij}
\end{align*}  
It is easily checked that $\eta^\pm$ and $\theta^\pm$ are comodule homomorphisms.

%
\subsection{Smash Product}
Suppose $H$ is a Hopf algebra and $B$ a right $H$-comodule algebra with structure map $\eta$ and $A=B^{\co \eta}$.  The algebra $B$ is  called a {\em  (right) $H$-cleft  extension} if there is a comodule map $\gamma: H \to B$ where $\gamma(1)=1$ which is convolution invertible (i.e., there is a linear map $\overline{\gamma}:H \to B$ so that $(\gamma*\overline{\gamma})(h)= \epsilon(h)\cdot 1$ and $(\overline{\gamma}*\gamma)(h)=\epsilon(h)\cdot 1$ for all $h \in H$ where $*$ is the convolution product.) Similarly one can define  a {\em left $H$-cleft extension} for $B$ a left $H$-comodule algebra. 

According to  a result of  \cite{Doi1986}, shown in  \cite{Mon1993} Proposition 7.2.3, if  $\sigma:H\otimes H \to A$  is defined by
\[
\sigma(h,k)=\sum_{(h),(k)} \gamma(h_1)\gamma(k_1)\overline{\gamma}(h_2k_2)
\]
and a left $H$-action on $A$ is given by
 \[
 h.a=\sum_{(h)}\gamma(h_1)a\overline{\gamma}(h_2),
 \]
 then the {\em (right) crossed product  $A\#_\sigma H$} is a $k$-algebra with multiplication defined by
 \[
 (a\#h)(b\#k)=\sum_{(h),(k)} a(h_1.b)\sigma(h_2,k_1)\#h_3k_2
 \]
for all $h,k \in H$ and $a,b \in A$.  
Moreover,   there is a $k$-algebra isomorphism $\Phi: A\#_\sigma H \to B$ given by $\Phi(a\#h)=a\gamma(h)$. Similar results hold for a {\em left crossed product $H\#_\sigma A$}  by making the appropriate modifications to the above.

If $\gamma$ is  a $k$-algebra homomorphism then   $\sigma$ is trivial.  For in this case we get
\[
\sigma(h,k)=\sum_{(h),(k)} \gamma(h_1k_1)\overline{\gamma}(h_2k_2)=(\gamma*\overline{\gamma})(hk)=\epsilon(hk)\cdot 1
\]
for all $h, k \in H$.  In this case, $A\#_\sigma H = A\#H$ is a {\em smash product} with multiplication
$$(a\#h)(b\#k)=\sum_{h} (a(h_1.b))\#h_2k$$

\subsection{$O_q(N^\pm)$ as Coinvariants}
\begin{thm}
\label{smash product}
Let $A^\pm=O_q(B^\pm)^{\co \eta^\pm}$ then $A^\pm\# O_q(T)\cong O_q(B^\pm)$. Similarly letting $C^\pm=O_q(B^\pm)^{\co \theta^\pm}$ then $O_q(T)\#C^\pm \cong O_q(B^\pm)$.
\end{thm}
\begin{proof}
Let $r^\pm: O_q(T)\to O_q(B^\pm)$ be the $k$-algebra homomorphism such that $r^\pm(Y_{ii})=X_{ii}$.  
Define $\overline{r^\pm}: O_q(T)\to O_q(B^\pm)$ by  $\overline{r^\pm}=r^\pm S_T$.  Then 
\begin{align*}
(r^\pm*\overline{r^\pm})(Y)=\sum_{(Y)}r^\pm(Y_1)\overline{r^\pm}(Y_2)=r^\pm\left(\sum_{(Y)} Y_1S_T(Y_2)\right)=r^\pm(\epsilon_T(Y)\cdot 1)=\epsilon(Y)\cdot 1
\end{align*}
for all $Y \in O_q(T)$. Similarly, $(\overline{r^\pm}*r^\pm)(Y)=\epsilon(Y)\cdot 1$. Hence, $\overline{r^\pm}$ is the convolution inverse of $r^\pm$. 

To check that $r^\pm$ is a right $O_q(T)$-comodule map we need to show that $\eta^\pm r^\pm=(r^\pm\otimes id)\Delta_T$.  Since $r^\pm$ and $\eta^\pm$ are $k$-algebra homomorphisms, it is sufficient to show the equality holds on the $Y_{ii} \in O_q(T)$.
This holds because
$$\eta^\pm r^\pm(Y_{ii})=\eta^\pm(X_{ii})=X_{ii}\otimes Y_{ii}$$
and 
$$(r^\pm\otimes id)\Delta_T(Y_{ii})=(r^\pm\otimes id)(Y_{ii}\otimes Y_{ii})=X_{ii}\otimes Y_{ii}$$
Hence, $r^\pm$ are right $O_q(T)$-comodule homomorphisms. Therefore $O_q(B^\pm)$ is an $H$-cleft extension.  Moreover, by the above discussion, using the result of \cite{Doi1986}  there is a $k$-algebra isomorphism $\Phi^\pm: A^\pm\# O_q(T) \to O_q(B^\pm)$ where $\Phi^\pm(X\#Y)=Xr^\pm(Y)$.

Similarly, by making the appropriate changes to the above proof,  there is a $k$-algebra isomophrism $\Psi^\pm:O_q(T)\#C^\pm \to O_q(B^\pm)$ defined by $\Psi(Y\#X)=r^\pm(Y)X$. \end{proof}

 It is natural to ask what are the coinvariants for $O_q(B^\pm)$ using the structure maps $\eta^\pm$ or $\theta^\pm$?  
Note that for for any $X_{jj}^{-1}, X_{ij} \in O_q(B^\pm)$
\begin{align*}
\eta^\pm(X_{ij}X_{jj}^{-1})=\eta^\pm (X_{ij})\eta^\pm(X_{jj}^{-1})=(X_{ij}\otimes Y_{jj})(X_{jj}^{-1}\otimes Y_{jj}^{-1})=X_{ij}X_{jj}^{-1}\otimes 1
\end{align*}
 That is $X_{ij}X_{jj}^{-1}\in O_q(B^\pm)^{\co \eta^\pm}$. Similarly, $X_{ii}^{-1}X_{ij}\in O_q(B^\pm)^{\co \theta^\pm}$.  Since $\eta^\pm$ and $\theta^\pm$ are algebra homomorphism this  implies that  $O_q(N^\pm)'$ is a subalgebra of $O_q(B^\pm)^{\co \eta^\pm}$ and $O_q(N^\pm)$ is a subalgebra of $O_q(B^\pm)^{\co \theta^\pm}$. 
In fact  these algebras are exactly the coinvariants for $\eta^\pm$ and $\theta^\pm$  which we now show.

\begin{thm}
$O_q(N^\pm)'=O_q(B^\pm)^{\co \eta^\pm}$ and $O_q(N^\pm)= O_q(B^\pm)^{\co \theta^\pm}$.
\label{N is coinvariants}
\end{thm}
\begin{proof}
From our discussion we have shown that $O_q(N^\pm)'\subseteq O_q(B^\pm)^{\co \eta^\pm}$.  Since the map $\Phi^\pm$ from Theorem \ref{smash product} is an isomorphism  of $A^\pm \otimes O_q(T)$ onto $O_q(B^\pm)$, it is sufficient to show that $\Phi^\pm$ maps $O_q(N^\pm)' \otimes O_q(T)$ onto $O_q(B^\pm)$. 

Notice  that
\[
\Phi^\pm(O_q(N^\pm)'\otimes O_q(T))=O_q(N^\pm)'r^\pm(O_q(T)) \subseteq O_q(B^\pm)
\]
Therefore we need only show that $O_q(N^\pm)r^\pm(O_q(T))$  is a subalgebra that contains all the $X_{ij}$ to prove the proposition. 

Let $L^\pm$ be  the subalgebra  generated by $\{X_{ii}^{\pm 1}\mid 1\leq i\leq n+1\}$  in $O_q(B^\pm)$. Note that the image of the maps $r^\pm$  from Theorem \ref{smash product} is $L^\pm$.   Using the projection homomorphism $p^\pm$,  it is straightforward to check  that $p^\pm  r^\pm=id_{O_q(T)}$ and $r^\pm  p^\pm=id_{L^\pm}$. Hence $r^\pm$ is an isomorphism from   $O_q(T)$ to $L^\pm$.  Since each of the $X_{kk}$ commutes up to a scalar with each of the $X_{ij} \in O_q(B^\pm)$ we have that $O_q(N^\pm)'r^\pm(O_q(T))$ is a subalgebra of $O_q(B^\pm)$.

Now for $X_{jj}^{-1}, X_{ij} \in O_q(B^\pm)$ with $i\neq j$ we have 
$$\Phi^\pm(X_{ij}X_{jj}^{-1} \otimes Y_{jj})=X_{ij}X_{jj}^{-1}r^\pm(Y_{jj})=X_{ij}X_{jj}^{-1}X_{jj}=X_{ij}$$
Hence, $O_q(N^\pm)'r^\pm(O_q(T)) = O_q(B^\pm)$.  Therefore $O_q(N^\pm)'=O_q(B^\pm)^{\co \eta^\pm}$.

In the same way one can use the isomorphism $\Psi^\pm$ to show $O_q(N^\pm)= O_q(B^\pm)^{\co \theta^\pm}$.
\end{proof}

Finally we have the following corollary.
\begin{cor}
$O_q(B^\pm)\cong O_q(T)\# O_q(N^\pm) \cong O_q(N^\pm)\# O_q(T)$.
\end{cor}
\begin{proof}
The first isomorphism follows directly from Theorem \ref{smash product} and Theorem \ref{N is coinvariants}.  For the second isomorphism we note that Theorem \ref{smash product} and Theorem \ref{N is coinvariants} imply that $O_q(B^\pm) \cong  O_q(N^\pm)'\# O_q(T)$.  Since from Theorem \ref{OqIsomToOq'} we have $O_q(N^\pm) \cong O_q(N^\pm)'$ the second isomorphism holds.
\end{proof}

\section{$C_q$: The Algebra of Matrix Coefficients}
Having investigated $O_q(B^\pm)$ and $O_q(N^\pm)$ we now switch to  the quantum function algebra or algebra of matrix coefficients $C_q$. In order to do this though we first need some background.

\subsection{Root and Weight Lattices}
Let $C=(c_{ij})$ be the Cartan matrix for the Lie algebra $\mathfrak{sl}_{n+1}$.  Let $\Phi$ be a set of roots for the Lie algebra and  $\Pi=\{\alpha_1,\cdots, \alpha_n\}$ be some choice of positive simple roots. Define 
\begin{align*}
\Q\Pi&:=\left\{ \sum_{i=1}^n r_i\alpha_i \mid r_i \in \Q,\, \alpha_i \in \Pi\right\}& &
\Z\Pi:=\left\{ \sum_{i=1}^n k_i\alpha_i \mid k_i \in \Z ,\, \alpha_i \in \Pi\right\}
\end{align*}
 called  the {\em rational root lattice} and  {\em root lattice} respectively.   
There is a  nondegenerate bilinear form $(-_,-)$ on $\Q\Pi$  determined by $(\alpha_i,\alpha_j)=c_{ij}$.  
There exists $\omega_i \in \Q\Pi$ called a {\em fundamental weight} with the property that  $(\omega_i,\alpha_j)=\delta_{ij}$ for all $\alpha_j\in \Pi$.  In the present case of $\frsl_{n+1}$ we have  

\begin{align*}
\omega_i=\frac{1}{n+1}\big((n-i+1)\alpha_i&+2(n-i+1)\alpha_2+\ldots (i-1)(n-i+1)\alpha_{i-1}\\ &+i(n-i+1)\alpha_i+i(n-i)\alpha_{i+1}+\ldots +i\alpha_n
\big)
\end{align*}
Therefore using the convention that  $\omega_0=\omega_{n+1}=0$ we have
\[
\alpha_i=-\omega_{i-1}+2\omega_i-\omega_{i+1}
\]
Define 
\begin{align*}
\Lambda:=\left\{\sum_{i=1}^n k_i\omega_i \;|\; k_i \in \Z\right\}& &
\Lambda^+:=\left\{\sum_{i=1}^n k_i\omega_i \;|\; k_i \in \Z_{\geq 0}\right\}
\end{align*} 
These are called the {\em weight lattice} and {\em positive weight lattice} respectively.
Note that  $\Z\Pi\subset \Lambda$ and that for $\lambda \in \Lambda$ and $\mu \in \Z\Pi$  we have $(\lambda, \mu)\in \Z$.

Finally,  there is a partial order on $\Lambda$ by defining  $\lambda\geq 0$ if there exists  $k_i \in \Z_{\geq 0}$ so that 
$\lambda=k_1\alpha_1+k_2\alpha_2+\cdots+k_n\alpha_n$. Hence, using this definition we have for $\mu,\lambda \in \Lambda$ that $\mu \geq\lambda$ if and only if $\mu-\lambda \geq0$.

\subsection{$U_q(\frsl_{n+1})$ : Structure and  Subalgebras}
We now give a presentation of the Hopf algebra $U_q(\frsl_{n+1})$ found in \cite{Jan1995} section 4.3 as well as  some other important algebras  related to it.

 Denote by $U_q=U_q(\frsl_{n+1})$  the algebra generated by the elements $K_\lambda$  for $\lambda \in \Z\Pi$ and  $E_{\alpha_i}, F_{-\alpha_i}$  where $\alpha_i \in \Pi$ presented with the following relations:

\begin{align*}
    K_\lambda K_{-\lambda}&=K_0=1                                                                       &     K_\lambda K_\mu&=K_\mu K_\lambda =K_{\lambda+\mu}\\
     K_\lambda E_{\alpha_i}=&q^{(\lambda,\alpha_i)}E_{\alpha_i} K_\lambda          & K_\lambda F_{-\alpha_i}&=q^{(\lambda,-\alpha_i)}F_{-\alpha_i} K_\lambda\\
E_{\alpha_i}E_{\alpha_j}=&E_{\alpha_j}E_{\alpha_i} \;\mathrm{for}\; |i-j|>2&
F_{-\alpha_i}F_{-\alpha_j}&=F_{-\alpha_j}F_{-\alpha_i} \;\mathrm{for}\; |i-j|>2
\end{align*}
\vspace{-.28in}
\begin{align*}
 [E_{\alpha_i},F_{-\alpha_j}]=\delta_{ij}&\frac{K_{\alpha_i}-K_{-\alpha_j}}{q-q^{-1}}  \\
E_{\alpha_i}^2E_{\alpha_j}-(q+q^{-1})E_{\alpha_i}E_{\alpha_j}E_{\alpha_i}+&E_{\alpha_j}E_{\alpha_i}^2=0 \;\mathrm{for}\; |i-j|=1\\
F_{-\alpha_i}^2F_{-\alpha_j}-(q+q^{-1})F_{-\alpha_i}F_{-\alpha_j}F_{-\alpha_i}+&F_{-\alpha_j}F_{-\alpha_i}^2=0 \;\mathrm{for}\; |i-j|=1
\end{align*}

We  use the notation  $E_i=E_{\alpha_i}$, $F_i=F_{-\alpha_i}$, and $K_i=K_{\alpha_i}$.  Moreover   $K_{\alpha_i}$ is invertible with inverse $K_{-\alpha_i}$ therefore $K^{-1}_i=K_{-\alpha_i}$. 

$U_q$ is  a Hopf algebra with comultiplication $\Delta$, counit $\epsilon$, and antipode $S$ determined by the following:
\begin{align*}
\Delta(E_i)&=K_i\otimes E_i +E_i\otimes 1& &\epsilon(E_i)=0& &S(E_i)=-K_i^{-1}E_i\\
\Delta(F_i)&=F_i\otimes K_i^{-1}+1\otimes F_i& &\epsilon(F_i)=0& & S(F_i)=-F_iK_i\\
\Delta(K_\lambda)&=K_\lambda\otimes K_{\lambda}& & \epsilon(K_\lambda)=1& & S(K_\lambda)=K_{-\lambda}
\end{align*}

Denote by
$U_q^{\geq 0}$ (resp. $U_q^{\leq 0}$) the subalgebra of $U_q$ generated by the $K_\lambda$ for $ \lambda \in \Z\Pi$ and $E_{i}$ (resp. $F_i$) for $i=1,2,\ldots, n$. It is a sub Hopf algebra  of $U_q$ and is called the {\em quantized positive (resp. negative)  Borel subalgebra } of $U_q$.  
Denote by $U_q^0$ the subalgebra of $U_q$ generated by $K_\lambda$ where $\lambda \in \Z\Pi$.  This algebra is called the {\em quantum torus of $U_q$.} Finally, denote by  $U_q^+$ (resp. $U_q^-$) the subalgebra generated by the $E_i$ (resp. $F_i$). This is referred to as the {\em positive (resp.  negative)  nilpotent subalgebra of $U_q$}.

We will also be interested in a slightly larger algebra  than $U_q$  denoted by $\check{U}_q$.  Specifically, $\check{U}_q$ is generated by  $K_\lambda$ for $\lambda \in \Lambda$ instead of just $\lambda \in \Z\Pi$ and the $E_i$,  $F_i$, and is  presented with the same relations  above where we now allow $\lambda \in \Lambda$.  It is also a Hopf algebra with the same description as the Hopf algebra structure for $U_q$ again allowing $\lambda \in \Lambda$.

The positive (resp. negative) Borel subalgebra of $\check{U}_q$  denoted by $\check{U}_q^{\geq0}$ (resp. $\check{U}_q^{\leq 0}$) are defined analogously as above.  The quantum nilpotent subalgebra $\check{U}_q^+$ (resp. $\check{U}_q^-$), that is, the subalgebra generated by the $E_i$ (resp. $F_i$), is exactly the same as $U_q^+$ (resp. $U_q^-$).


\subsection{Weights and Weight Spaces}
Let $V$ be a left $U_q$-module.   If for a nonzero vector $v\in V$ there  is a   $\lambda \in \Lambda$ and homomorphism $\sigma\in \hom (\Z\Pi, \{\pm1\})$ so that  for all $\mu \in \Z\Pi$ we have
$$K_{\mu}v=\sigma(\mu) q^{(\lambda, \mu)}v,$$
we call $v$ a  {\em weight vector} with {\em weight $\lambda$} of {\em  type $\sigma$}.  We denote the set of weights of $V$ by $\Omega(V)$. 
The set
 $$V_{(\lambda,\sigma)}=\big\{ v\in V \mid K_\mu v=\sigma(\mu)q^{(\lambda,\mu)}v\; \mathrm{for \; all }\; K_\mu \in U_q^0\big\}$$
is a subspace of $V.$ We call $V_{(\lambda, \sigma)}$ a {\em weight space} of $V$.  For the special case where  $\sigma= {\bf 1}$  (that is, $\sigma(\mu)=1$ for all $\mu \in \Lambda$) we denote    $V_{(\lambda,\mathbf{1})}$ by $V_\lambda$.

If  a weight vector $v \in V$  has the property that  $E_iv=0$ for all  the $E_i$ then $v$ is called a {\em maximal weight vector} for $V$.  Analogously, a  weight vector $v\in V$ is a {\em minimal weight vector} if $F_iv=0$ for all the $F_i$.

For  $\sigma \in \hom(\Z\Pi, \pm1)$ denote $V_\sigma=\sum_{\lambda \in \Omega(V)} V_{(\lambda,\sigma)}$.
Proposition 5.1 in  \cite{Jan1995}   implies that any $U_q$-module can be written 
 $$V=\bigoplus _{\sigma} V_\sigma
 $$
  Hence, if $V$ is an irreducible finite dimensional module then $V=V_\sigma$ and so $V$ has a well defined type $\sigma$.

\subsection{The Category  $\mathscr{C}_q$}
Define the category $\mathscr{C}_q$ to be the full subcategory of \nobreak{$U_q$-modules} whose objects are the  finite dimensional  left $U_q$-modules of type {\bf 1}.   
We note that for  $V$ and $W$  objects in $\mathscr{C}_q$ and    $V_\lambda$ and $W_\mu$  weight spaces  then $V_\lambda\otimes W_\mu \subseteq (V\otimes W)_{\lambda +\mu}$ and $(V_\lambda)^*=(V^*)_{-\lambda}$. (\cite{Jan1995} section 5.3)

 We collect some basic facts about  $\mathscr{C}_q$ which can be found in \cite{Jan1995} chapters 5 and 6.

\begin{thm}
\label{V is direct sum of weight spaces}
If $V\in \mathscr{C}_q$ then $V$ is the direct sum of its weight spaces.  
\end{thm}
\begin{thm}
\label{irreducible modules}
Let $V\in \mathscr{C}_q$ be  irreducible. 
\begin{enumerate}[(i)]
\item $V$ contains a maximal weight vector $v_\lambda$ with weight $\lambda \in \Lambda^+$ .
\item $\dim V_\lambda=1$. 
\item If $\nu$ is any other weight for $V$ then $\nu<\lambda$. \item $U_q^{\leq 0}v_\lambda=V$.
\end{enumerate}

$V$ also  contains a minimal weight vector $v_{-\mu}$ with weight  $-\mu\in -\Lambda^+$  with analogous properties.
\end{thm}

From  Theorem \ref {irreducible modules}, if  $\lambda$ is  the weight of a maximal weight vector $v_\lambda$ then we call  $v_\lambda$   {\em a highest  weight vector} and $\lambda$  {\em  the highest weight} of $V$.     Analogously  $v_{-\mu}$ is called {\em a lowest weight vector} and $-\mu$ called  {\em the lowest weight} of $V$.  

\begin{thm}
\label{Every lambda there is a highest weight}
For every $\lambda \in \Lambda^+$ there is an irreducible module $V \in \mathscr{C}_q$ which has  highest weight $\lambda$.  Moreover, if $W$ is any other irreducible module in $\mathscr{C}_q$ with highest weight  $\mu$, then $V\cong W$ if and only if $\mu =\lambda$. 
\end{thm}
From Theorem \ref{Every lambda there is a highest weight} we denote an irreducible module with highest weight $\lambda$  by $V(\lambda)$ and, using this notation, we have   $V(\mu)\cong V(\lambda)$ if and only if $\mu=\lambda$.

\subsubsection{The Module $V(\omega_1)$}
Since the module $V(\omega_1)$ is so integral for what follows, we examine it a bit more closely.  

It is known that  $\dim V(\omega_1)=n+1$ (\cite{Jan1995} section 5.15)  and there exists a  basis  $\{e_1, e_2, \ldots, e_{n+1}\}$  so that each $e_j$ is a weight vector  with  weight $\omega_1-\alpha_1+\cdots-\alpha_{j-1}$. Moreover, we can select the $e_j$ so that  $E_ie_j=\delta_{i,j-1}e_{j-1}$ and $F_ie_j=\delta_{i,j}e_{j+1}$. The basis is well-ordered with respect to weight and  $e_1$ has highest weight $\omega_1$.  Also note that  $i\leq j$ if and only if  $\mathrm{wt}\; e_i\geq \mathrm{wt}\; e_j$.     

 In $\mathfrak{sl}_{n+1}$  since  $\alpha_i=-\omega_{i-1}+2\omega_i-\omega_{i+1}$ for each $i=1,2,\ldots,n$, using the convention $\omega_0=\omega_{n+1}=0$,  we have the weight of $e_j$ is $\beta_j:=-\omega_{j-1}+\omega_{j}$.  Hence 
 for $\lambda \in \Z\Pi$ and $1\leq j\leq n+1$   we get
\begin{equation}
K_\lambda e_j=q^{(\beta_j,\lambda)}e_j \label{Kaction}
\end{equation}

Moreover, note that if $I=({i_1},i_2,\cdots i_{N-1},{i_N})$  where each $i_k \in \{1,2,\ldots, n\}$, adopting the convention 
\begin{align*}
E_I=E_{i_1}E_{i_2}\cdots E_{i_{N-1}}E_{i_N}& &\mathrm{and}& &F_I=F_{i_1}F_{i_2}\cdots F_{i_{N-1}}F_{i_N}
\end{align*}
we have 
\begin{equation}\label{Eaction} E_Ie_{j}=\delta_{i_1,j-N+1}\delta_{i_2,j-N+2}\cdots \delta_{i_{N-1},j-2}\delta_{i_N,j-1}e_{i_1}\end{equation}
Therefore if $i< j$ we have $E_Ie_j=e_i$ if and only if $I=(i,{i+1},\cdots, {j-2},{j-1})$  and $E_I e_j=e_j$ if and only if $I=\emptyset$. Furthermore, if  $I$ is any other finite sequence of elements from $\{1,2,\ldots, n+1\}$ then  $E_Ie_j=0$ for all $j$.

\subsection{ Definition and Structure  of $C_q$}
For each $V\in \mathscr{C}_q$ and   $f \in V^*$ and $v \in V,$ we define the linear functional $c^V_{f,v}\in U_q^*$ by $c^V_{f,v}(u)=f(uv)$ for all $u \in U_q$.
 We call  a functional $c^V_{f,v}$ a {\em coordinate function} or a {\em matrix coefficient for $V$}. Define
 $$C_q:=\left\{c^V_{f,v} \mid V\in \mathscr{C}_q,\, f \in V^*, \,v \in V\right\}$$ 
 Since the annihilator of $V$ is contained in the kernel of $c^V_{f,v}$ and since $V$ is finite dimensional, $c^V_{f,v}\in U_q^\circ$ where $U_q^\circ$ is the Hopf dual of $U_q$. 
Hence, $C_q \subseteq U_q^\circ$.  

Note that using the standard addition and multiplication from $U_q^\circ$, we have
 \begin{align*}
c^V_{f,v}+c^W_{g,w}&=c^{V\oplus W}_{f\oplus g,v\oplus w}& &
c^V_{f,v}c^W_{g,w}=c^{V\otimes W}_{f\otimes g,v\otimes w}
\end{align*}
 Since $\mathscr{C}_q$ is closed under finite direct sums and tensor products (\cite{Jan1995} section 5.4) it follows that $C_q$ is a $k$-subalgebra of $U_q^\circ$. 

For the irreducible modules $V(\mu)$ in $\mathscr{C}_q$, we will often denote 
 $c^{V(\mu)}_{f,v}$ simply by $c^\mu_{f,v}$. Moreover, for $v\in V(\mu)$ and $f\in V(\mu)^*$ weight vectors  with weights $\nu $ and $-\lambda$ respectively,  we may denote $c^{\mu}_{f,v}$ by either $c^\mu_{f,\nu}$ or $c^\mu_{-\lambda,\nu}$.     Hence, any proposition written with $c^\mu_{-\lambda,\nu}$ will be independent of any choice of weight vectors $v$ or $f$.

 Denote the natural dual basis $\left\{f^1,f^2,\ldots, f^{n+1}\right\}$ for $V(\omega_1)^*$ corresponding to $\{e_1,e_2,\ldots, e_{n+1}\}$ so that $f^i(e_j)=\delta_{ij}$.  The next theorem follows from  \cite{Hod1993} Theorem 1.4.1.

\begin{thm}
\label{Cq and Oq isom}
$\left\{c^{\omega_1}_{f^i,e_j}\mid i,j =1,2, \ldots n+1\right\}$ is a generating set for the algebra $C_q$.
Moreover, there exists a $k$-algebra isomorphism $\kappa: O_q(\SL_{n+1}) \to C_q$ so that    $\kappa(X_{ij})=c^{\omega_1}_{f^i,e_j}$.
\end{thm}
Using the isomorphism from  Theorem \ref{Cq and Oq isom} we will often abuse notation and denote $c^{\omega_1}_{f^i,e_j}$ by $X_{ij}$.

\section{The Restriction Map and  $C_q(B^\pm)$}
For any element $c\in C_q$ we may restrict the domain of $c$ to the subalgebra ${U}_q^{\geq0}$ (resp. $U_q^{\leq0}$).  This induces a well-defined $k$-algebra homomorphism from  $C_q$ to $(U_q^{\geq0})^*$ (resp. $(U_q^{\leq0})^*$). Let $\rho^+$ (resp. $\rho^-$) be the restriction homomorphism and  denote $\im \rho^+$ by $C_q(B^+)$ (resp. $\im \rho^-$ by $C_q(B^-)$).  
For $c\in C_q$ we denote $\rho^+(c)$ by $\overline{c}$. From this convention we note that   $\overline{c} \neq 0$ if and only if $c(u)\neq 0$ for some $u \in U_q^{\geq 0}$.

We now wish to better understand  $\rho^+(c^{\omega_1}_{f^i,e_j})$ which, using the conventions above, we denote by $\overline{X}_{ij}$.

\begin{lem}
\label{Xij map}
For $\lambda \in \Z\Pi$ and $I$ a finite sequence of elements from $\{1,2,\ldots, n+1\}$ we have the following:

For $i\leq j$,
$$\overline{X}_{ij}(K_\lambda E_I)=\threepartdef{q^{(\beta_i,\lambda)}}{I=\emptyset\; \mathrm{and} \; i=j}{q^{(\beta_i,\lambda)}} {I=(i,{i+1},\cdots, {j-2},{j-1})\; \mathrm{and} \; i<j}{0}{\mathrm{otherwise}}$$
where $\beta_i=-\omega_{i-1}+\omega_i$.

For $i>j$, then $\overline{X}_{ij}(K_\lambda E_I)=0$.
\end{lem}

\begin{proof}
By  equation \eqref{Eaction}  if $i< j$ we have $E_Ie_j=e_i$ if and only if $I=(i,{i+1},\cdots,{j-2}, {j-1})$  and $E_I e_j=e_j$ if and only if $I=\emptyset$. Furthermore, if  $I$ is any other finite sequence of elements from $\{1,2,\ldots, n+1\}$ then  $E_Ie_j=0$ for all $j=1,2,\ldots,n+1$.
Hence by equation \eqref{Kaction} if $i<j$ and  $I=(i,\cdots, {j-1})$ we have $K_\lambda E_Ie_j=q^{(\beta_i,\lambda)}e_i$ and if $I=\emptyset$ then $K_\lambda E_Ie_j=q^{(\beta_i,\lambda)}e_j$.  Thus if  $i\leq j$ then $\overline{X_{ij}}(K_\lambda E_I)$ is as indicated.

Note that for all the $E_i$ we have  $E_{i}e_j$ is either  0 or a weight vector of lower weight.  Hence, for $i>j$ there is no $E_I$ so that  $E_Ie_j=e_i$.   Therefore for $i>j$ we have $f^i(E_Ie_j)=0$ for all $E_I$.  Hence, $\overline{X}_{ij}(K_\lambda E_I)=0$ for all $\lambda \in \Z\Phi^+$ and $I$ a finite sequence with entries in $\{1,2,\ldots, n+1\}$.
\end{proof}

 Let $W$ be the Weyl group of $\frsl_{n+1}$ and denote by $w_0$ the longest word for $W$.  Define the following  ideals of $C_q$ 
 \begin{align*}
J^+:&=\left\langle c^\mu_{-\alpha,\mu}\mid \alpha< \mu\;\mathrm{where}\; \mu \in \Lambda^+, \;\alpha\in \Omega(V(\mu))  \right \rangle\\
J^-:&=\left\langle c^\mu_{-\alpha,w_0\mu}\mid \alpha> w_0\mu\;\mathrm{where}\; \mu \in \Lambda^+,\; \alpha\in \Omega(V(\mu))  \right \rangle
\end{align*}
 The following two results are  from  9.2.11 and 9.2.12 in  \cite{Jos1995}.
\begin{thm}
\label{J+}
$J^+=\langle X_{ij} \;|\; i>j\rangle=\ker \rho^+$ and $J^-=\left\langle X_{ij}\mid i<j\right\rangle=\ker \rho^-$.
\end{thm}

\begin{cor}
\label{OqB+ cong CqB+}
The restriction maps $\rho^\pm$ induce  $k$-algebra isomorphisms of $O_q(B^\pm)$ onto $C_q(B^\pm)$.
\end{cor}

\section{The Dual Pairing}
By \cite{Fra1998}  Corollary 3.3, if $\sqrt[n+1]{q} \in k$ there exists a unique nondegenerate bilinear pairing  $(-,-):\check{U}_q^{\leq 0} \times \check{U}_q^{\geq 0} \to k$ defined by the following properties: 
for all $u,u'\in \check{U}^{\geq0}$ and all $v, v' \in \check{U}^{\leq 0}$ and all $\mu, \nu \in \Lambda$ and $\alpha, \beta \in \Pi$
\begin{align*}
(u,vv')&=(\Delta(u),v'\otimes v)& (uu',v)&=(u\otimes u',\Delta(v))\\
(K_\mu, K_\nu)&=q^{-(\mu, \nu)}& (F_\alpha,E_\beta)&=-\delta^{}_{\alpha \beta}\qhat^{\, -1}\\
(K_\mu,E_\beta)&=0& (F_\alpha,K_\mu)&=0
\end{align*}
We note there is a pairing on $U_q^{\leq 0}\times U_q^{\geq 0}$ defined in \cite{Jan1995} 6.12 as well as \cite{Jos1995}  9.2.10 with the same properties as above. 

The bilinear pairing above can be restricted to a bilinear pairing on  $\check{U}_q^{\leq 0} \times {U}_q^{\geq 0}$. In fact the assumption that  $\sqrt[n+1]{q} \in k$ is not needed to have a well defined dual pairing on $\check{U}_q^{\leq 0} \times {U}_q^{\geq 0}$.   Since we will  primarily be studying this  pairing,  we remove this assumption. We now show that the pairing in nondegenerate by  slightly modifying  the proof in \cite{Fra1998} Corollary 3.3.

\begin{thm}
\label{nondeg dual pairing}
The dual pairing on $\check{U}_q^{\leq 0} \times {U}_q^{\geq 0}$ is nondegenerate.
\end{thm}
\begin{proof}
 By \cite{Jan1995} 4.7, $U^+_q$ and $U^-_q$ are $\Z\Pi$-graded with $\deg E_\alpha=\alpha$ and $\det F_{-\alpha}=-\alpha$ for $\alpha \in \Pi$.   For each $\mu \in \Z\Pi$ denote a basis of $U_\mu^+$ by $\{u_i^\mu\}$. By \cite{Jan1995} Corollary 8.30, the pairing when restricted to $U^-_{-\mu}\times U^+_\mu$ with $\mu \in \Z\Pi$ is nondegenerate.   Therefore, we may select  a corresponding dual basis of $U_{-\mu}^-$,  which we denote by $\{v_i^{-\mu}\}$, with the property that $(v_i^{-\mu},u_j^{\mu})=\delta_{ij}$.

Suppose that $y \in \check{U}_q^{\leq 0}$ so that $(y,x)=0$ for all $x \in U_q^{\geq 0}$. We may write  
$$y=\sum_{\mu'\in \Z\Pi, i} v^{-\mu'}_i p_{\mu',i}(K)$$ where 
 $$p_{\mu',i}(K)=\sum_{\lambda' \in \Lambda}c^{\mu',i}_{\lambda'}K_{\lambda'}$$
for some scalars $c^{\mu',i}_{\lambda'}$.  It follows from \cite{Jan1995} 6.10(3) and 6.10(4) that if $\mu \in \Z\Pi$ then for  all $\eta \in \Z\Pi$ and all $j$ we have 
\[
0=(y,u^{\mu}_jK_\eta)=\sum_{\lambda'\in \Lambda}c^{\mu,j}_{\lambda'}q^{-(\eta,\lambda')}
\]
Notice that for all $\lambda' \in \Lambda$ there exist scalars $\lambda'_1,\ldots, \lambda'_n \in \Z$ so that $\lambda'=\lambda'_1 \omega_1+\cdots+\lambda'_n\omega_n$. Therefore, for  each $m=(m_1,\ldots,m_n)\in \Z^n$, setting $\eta= -(m_1\alpha_1+\cdots+m_n\alpha_n)$ we get
 \[
0=\sum_{\lambda'\in \Lambda}c^{\mu,j}_{\lambda'}q^{-(\eta,\lambda')}=\sum_{\lambda'\in \Lambda}c^{\mu,j}_{\lambda'} q^{m_1 \lambda'_1}\cdots q^{m_n\lambda'_n}=p_{\mu,j}(q^{m_1},\ldots,q^{m_n})=p_{\mu,j}(q^{m})
\]
where
$$p_{\mu,j}(x_1,\ldots,x_n)=\sum_{\lambda' \in \Lambda}c^{\mu,j}_{\lambda'}x_1^{\lambda'_1}\cdots x_n^{\lambda'_n}$$
Since $q$ is not a root of unity, it follows from \cite{Fra1998}  Lemma 3.2 that $p_{\mu,j}=0$ for all $\mu \in \Z\Pi$ and all $j$.  Hence, $y=0$.  

Suppose for $x\in U_q^{\geq 0}$ that  $(y,x)=0$ for all $y \in \check{U}_q^{\leq 0}$.  Since $x\in \check{U}_q^{\geq0}$ it follows from the nondegeneracy of $(-,-)$ on $\check{U}_q^{\leq 0} \times \check{U}_q^{\geq 0}$ that $x=0$.
\end{proof}

Using this  dual pairing we define the map $\phi : \check{U}_q^{\leq 0}\to ({U}_q^{\geq0})^*$ by  
$$\phi(u)(v)=(u,v)$$
 The map $\phi$ is a $k$-algebra homomorphism and since the bilinear form is nondegenerate by Theorem \ref{nondeg dual pairing},  $\phi$ is injective.   

 For $I=(i_1,i_2,\ldots, i_N)$ define  $\mathrm{wt}\;I=\alpha_{i_1}+\alpha_{i_2}+\cdots+\alpha_{i_N}$.  We say that $J\subseteq I$ if $J$ is a subsequence of $I$. For $J\subseteq I$ we denote $I-J$ to be the sequence remaining when the subsequence $J$ is removed.  
 
 It is  shown in \cite{Jan1995} 6.8(8) and 6.12    that  $(F_I,E_J)=0$ unless $\mathrm{wt\;}I=\mathrm{wt\;}J$.  If follows that $\ker (\phi(F_I))$ has finite codimension for every $I$.  Similarly, $\ker \phi(K_\mu)$  has  codimension 1.  Hence $\phi(u) \in (U_q^{\geq0})^\circ$ for all $u \in \check{U}_q^{\leq0}$.

\begin{lem}
\label{Xii and Xii+1 in phi}
 For the map $\phi$ we have  $\phi(K_{-\beta_i+\alpha_i}F_i)=-q^{-1}\qhat^{\,-1}\overline{X}_{ii+1}$ and $\phi(K_{\mp \beta_i})=\overline{X}_{ii}^{\pm1}$.
\end{lem}

\begin{proof}
From \cite{Jan1995} 6.9 (2) it follows   for all $\mu, \lambda \in \Lambda$ that $(K_\mu, K_\lambda E_I)=0$ if and only if $I\neq \emptyset$. Using this fact we get
\begin{align*}
 \phi(K_{-\beta_i+\alpha_i}F_i)(K_\lambda E_i)
 &=(K_{-\beta_i+\alpha_i}F_i,K_\lambda E_i)
=\left(K_{-\beta_i+\alpha_i}\otimes F_i, \Delta(K_\lambda E_i)\right)\\
&=(K_{-\beta_i+\alpha_i}\otimes F_i,K_{\lambda +\alpha_i}\otimes K_\lambda E_i +K_\lambda E_i\otimes K_\lambda)\\
&=(K_{-\beta_i+\alpha_i},K_{\lambda+\alpha_i})(F_i,K_\lambda E_i)+(K_{-\beta_i+\alpha_i},K_\lambda E_i)(F_i,K_\lambda)\\
&=(K_{-\beta_i+\alpha_i},K_{\lambda+\alpha_i})(F_i,K_\lambda E_i)
=(K_{-\beta_i+\alpha_i},K_{\lambda+\alpha_i})(\Delta(F_i),E_i\otimes K_\lambda)\\
&=(K_{-\beta_i+\alpha_i},K_{\lambda+\alpha_i})(F_i\otimes K_{-\alpha_i}+1\otimes F_i,E_i \otimes K_\lambda)\\
&=(K_{-\beta_i+\alpha_i},K_{\lambda+\alpha_i})\left((F_i,E_i)(K_{-\alpha_i},K_\lambda)+ (1,E_i)(F_i,K_\lambda)\right)\\
&=(K_{-\beta_i+\alpha_i},K_{\lambda+\alpha_i})(F_i,E_i)(K_{-\alpha_i},K_\lambda)\\
&=-q^{-(-\beta_i+\alpha_i,\lambda+\alpha_i)}\qhat^{\,-1}q^{(\alpha_i,\lambda)}=-q^{-1}\qhat^{-1}q^{(\beta_i,\lambda)}
\end{align*}

Now suppose $I \neq (i)$. It follows from  \cite{Jan1995} 6.8 (1), for $\lambda \in \Z\Pi$ and $I$ a finite sequence of elements in $\{1,2,\ldots, n\}$   
\[
\Delta(K_\lambda E_I)=\sum_{J\subseteq I}c^{}_{I,J} K_\lambda K_JE_{I-J}\otimes K_\lambda E_J
\]
for some scalars, $c_{I,J}$.  Hence, we have
\begin{align*}
 \phi(K_{-\beta_i+\alpha_i}F_i)(K_\lambda E_I)
 &=(K_{-\beta_i+\alpha_i}F_i,K_\lambda E_I)=\left(K_{-\beta_i+\alpha_i}\otimes F_i, \Delta(K_\lambda E_I)\right)\\
&=\left(K_{-\beta_i+\alpha_i}\otimes F_i,\sum_{J\subseteq I}c^{}_{I,J} K_\lambda K_JE_{I-J}\otimes K_\lambda E_J\right)\\
&=\sum_{J\subseteq I} c^{}_{I,J} (K_{-\beta_i+\alpha_i},K_\lambda K_JE_{I-J})(F_i,K_\lambda E_J)
\end{align*}
It follows from \cite{Jan1995} 6.8 (7) that $(F_i,K_\lambda E_J)= 0$ if and only if $J\neq (i)$. Moreover, if $J=(i)$  then since $I \neq (i)$ we have that $I-J \neq \emptyset$. This  implies $(K_{-\beta_i+\alpha_i},K_\lambda K_JE_{I-J})=0$. Therefore, we have $ \phi(K_{-\beta_i+\alpha_i}F_i)(K_\lambda E_I)=0$. Therefore, we have 
\begin{align*}
 \phi(K_{-\beta_i+\alpha_i}F_i)(K_\lambda E_I)=\twopartdef{-q^{-1}\qhat^{\,-1}q^{(\beta_i,\lambda)}}{I=(i)}
 {0}{I\neq (i)}
\end{align*}

 Similarly,
 \begin{align*}
 \phi(K_{\mp\beta_i})(K_\lambda)
 =(K_{\mp\beta_i},K_\lambda) 
 =q^{\pm(\beta_i,\lambda)}
\end{align*}
Moreover, if $\lambda \in \Z\Pi$ and $I$ is a finite sequence of elements from $\{1,2,\ldots,n\}$ and $I \neq \emptyset$ we have
  \begin{align*}
 \phi(K_{\mp\beta_i})(K_\lambda E_I)
 &=(K_{\mp\beta_i},K_\lambda E_I)\\
 &=(K_{\mp\beta_i}\otimes K_{\mp\beta_i}, E_I\otimes K_\lambda)\\
 &=(K_{\mp\beta_i},E_I)(K_{\mp\beta_i},K_\lambda)\\
 &=0
\end{align*}
Hence,
  $$\phi(K_{\mp\beta_i})(K_\lambda E_I)=\twopartdef{q^{\pm(\beta_i,\lambda)}}{I=\emptyset}{0}{\mathrm{otherwise}}$$
  
 By Lemma~\ref{Xij map},  
 $$\overline{X}_{i,i+1}(K_\lambda E_I)=\twopartdef{q^{(\beta_i,\lambda)}}{I=(i)}{0}{\mathrm{otherwise}}$$
 and 
$$\overline{X}_{ii}^{\pm1}(K_\lambda E_I)=\twopartdef{q^{\pm(\beta_i,\lambda)}}{I=\emptyset}{0}{\mathrm{otherwise}}$$ 

Since $U_q^{\geq 0}$  is spanned by  the elements of the form $K_\lambda E_I$ where $I$ is a finite sequence of elements from $\{1,2,\ldots, n\}$ and  $\lambda \in \Z\Pi$ then  $\phi(K_{-\beta_i+\alpha_i}F_i)=-q^{-1}\qhat^{\,-1}\overline{X}_{i,i+1}$ and $\phi(K_{\mp\beta_i})=\overline{X}_{ii}^{\pm 1}$.  \end{proof}

The following theorem can be found in \cite{Jos1995} 9.2.12; however, in the present case we may prove it more simply.

\begin{thm}
The map $\phi$ is an isomorphism of $\check{U}_q^{\leq 0}$ onto $C_q(B^+)$.
\label{phi isomorphism}
\end{thm}
\begin{proof}
Note that  $\mp\beta_1=\mp\omega_1$ so then 
$$\phi(K_{\mp\omega_1})=\phi(K_{\mp\beta_1})=\overline{X}_{11}^{\pm1}\in C_q(B^+)$$
 Proceeding inductively, we see that if $\phi(K_{\mp\omega_{i-1}} )\in C_q(B^+)$ then since $\phi(K_{\mp\beta_i})=\overline{X}_{ii}^{\pm1} $ we have 
$$\phi(K_{\mp\omega_{i}})=\phi(K_{\mp\beta_i})\phi(K_{\mp\omega_{i-1}})=\overline{X}_{ii}^{\pm1}\phi(K_{\mp\omega_{i-1}})\in C_q(B^+)$$
It then follows that $\phi(K_\mu) \in C_q(B^+)$ for all $\mu \in \Lambda$.

Finally, we note that $\beta_i-\alpha_i=\beta_{i+1}$ for $1\leq i\leq n$.  Therefore we have for $1\leq i \leq n$ that
\begin{align*}
\phi(F_i)=\phi(K_{\beta_i-\alpha_i}K_{-\beta_i+\alpha_i}F_i)=\phi(K_{\beta_{i+1}})\phi(K_{-\beta_i+\alpha_i}F_i)\in C_q(B^+)
\end{align*}
  
 Since the  $K_{\mu}$ with $\mu \in \Lambda$ and the $F_i$ generate $\check{U}_q^{\leq 0}$, then  $\mathrm{im}\; \phi\subseteq C_q(B^+)$.   Conversely, since the $\overline{X}_{ii}^{\pm1}$ and $\overline{X}_{i,i+1}$ generate $C_q(B^+)$ it follows that $\im \phi =C_q(B^+)$.
 Since  the dual pairing is nondegenerate,  $\phi$  is   injective.   Hence, $\phi$ is a  $k$-algebra isomorphism. 
\end{proof}


\begin{cor}
\label{Uq+ isomorphism}
If $\omega:\check{U}_q^{\geq0} \to \check{U}_q^{\leq 0}$ is the Cartan homomorphism then $\phi\circ \omega$  restricted to $U_q^+$ induces  an isomorphism onto $\rho^+(\kappa(O_q(N^+)'))$.
\end{cor}

\begin{proof}
The Cartan homomorphism is  defined by $\omega(K_\lambda)=K_\lambda$ and $\omega(E_i)=F_i$. It is clear that $\omega$ is an isomorphism, so $\phi\circ \omega$ is an isomorphism. Hence  $\phi\circ \omega$  restricted to $U_q^+$ is an injective $k$-algebra homomorphism.  Moreover, from  Theorem \ref{phi isomorphism}  for $i=1,2,\ldots, n$ 
$$\phi\circ\omega(E_i)=\phi(K_{\beta_{i+1}})\phi(K_{-\beta_i+\alpha_i}F_i)=\overline{X}_{i+1,i+1}^{-1}\left(-q^{-1}\qhat^{\,-1}\overline{X}_{i,i+1}\right)=-\qhat^{\,-1}\overline{X}_{i,i+1}\overline{X}_{i+1,i+1}^{-1}$$
Since $\rho^+(\kappa(O_q(N^+)'))$ is generated by the  $\overline{X}_{i,i+1}\overline{X}_{i+1,i+1}^{-1}$ for $i=1,2,\ldots n$ then $\phi\circ \omega$ when restricted to $U_q^+$ is an isomorphism.
\end{proof}

\begin{thm}
\label{OqN isom Uqn}
There is a $k$-algebra isomorphism $\psi:U_q^+ \to O_q(N^+)'$ such that\break   $\psi(E_i)=-\qhat^{\,-1}X_{i,i+1}X_{i+1,i+1}^{-1}$ for $1\leq i\leq n$.
\end{thm}

\begin{proof}
This follows from Corollary \ref{Uq+ isomorphism} and the fact  that, $\rho^+$ and $\kappa$ are isomorphisms.
\end{proof}

We note that since the Cartan automorphism $\omega$ is an isomorphism of $U_q^+$ onto $U_q^+$ it also follows form Theorem \ref{OqN isom Uqn} that $U_q^- \cong O_q(N^+)'$.

\subsection{Conclusion}
In conclusion, we have shown that the algebras $O_q(N^\pm)=O_q(B^\pm)^{\co \theta^\pm}$, $O_q(N^\pm)'=O_q(B^\pm)^{\co \eta^\pm}$, and $U_q^\pm$ are all isomorphic from Theorem \ref{OqN isom Uqn}, Theorem \ref{N is coinvariants}, and Theorem \ref{OqIsomToOq'}.

\bibliography{sources}{}
\bibliographystyle{acm}
\begin{center}
{\scshape \noindent Andrew Jaramillo\\
Department of Mathematics\\
University of California, Santa Barbara\\
Santa Barbara, CA 93106\\}
\href{mailto:drewj@math.ucsb.edu}{drewj@math.ucsb.edu}
\end{center}
\end{document}